\documentclass[a4paper,reqno]{amsart}
\usepackage{amsmath,amsfonts,amscd,amssymb,graphicx}
\usepackage{algorithm}
\usepackage{algpseudocode}
\usepackage{amsthm,ifthen}
\usepackage{subfig}
\newcommand{\nz}{\mathbb{N}}       
\newcommand{\rz}{\mathbb{R}}       
\theoremstyle{plain}
	\newtheorem{SATZ}{Proposition}[section]   
	\newtheorem{THM}[SATZ]{Theorem}

	\theoremstyle{definition}

	\newtheorem{BEM}[SATZ]{Remark}
\numberwithin{figure}{section}
\numberwithin{table}{section}
\numberwithin{equation}{section}
\DeclareMathOperator{\Div}{div}
\DeclareMathOperator{\Curl}{curl}
\newcommand{\divT}{\Div_\CT}
\newcommand{\gradT}{\nabla_\CT}
\DeclareMathOperator{\goh}{O}
\DeclareMathOperator{\shps}{c}
\newcommand{\halb}{\frac 12}
\newcommand{\be}{\beta}

\newcommand{\la}{\lambda}

\newcommand{\om}{\omega}
\newcommand{\Om}{\Omega}
\newcommand{\vfi}{\varphi}

\newcommand{\pz}{\mathbb{P}}
\newcommand{\betrag}[2][]{\left\lvert #2 \right\rvert_{#1}} 
\newcommand{\norm}[2][]{\left\lVert #2 \right\rVert_{#1}} 
\newcommand{\jump}[2][E]{{\mathbb J}_{#1}(#2)}
\newcommand\Ve{\mathbf{e}}
\newcommand\Vf{\mathbf{f}}

\newcommand\Vn{\mathbf{n}}

\newcommand\Vt{\mathbf{t}}
\newcommand\Vv{\mathbf{v}}
\newcommand\Vu{\mathbf{u}}
\newcommand\Vw{\mathbf{w}}

\newcommand\CF{\mathcal{F}}

\newcommand\CM{\mathcal{M}}

\newcommand\CT{\mathcal{T}}
\newcommand\CU{\mathcal{U}}
\newcommand\CV{\mathcal{V}}

\title[Quasi-Optimal Crouzeix-Raviart Discretization of Stokes Equations]{A Quasi-Optimal Crouzeix-Raviart\\ Discretization of the Stokes Equations}
\author{R. Verf{\"u}rth}
\address{Ruhr-Uni\-ver\-si\-t{\"a}t Bo\-chum \\ Fa\-kul\-t{\"a}t f{\"u}r Ma\-the\-ma\-tik \\ D-44780 Bo\-chum \\ Germany}
\email{ruediger.verfuerth@rub.de}
\author{P. Zanotti}
\address{TU Dortmund \\ Fakult{\"a}t f{\"u}r Mathematik \\ D-44221 Dortmund \\ Germany}
\email{pietro.zanotti@tu-dortmund.de}
\date{\today}
\dedicatory{Christine Bernardi in memoriam}
\begin{document}
\begin{abstract}
We present a modification of the Crouzeix-Raviart discretization of the Stokes equations in arbitrary dimension which is quasi-optimal, in the sense that the error of the discrete velocity field in a broken $H^1$-norm is proportional to the error of the best approximation to the analytical velocity field. In particular, the velocity error is independent of the pressure error and the discrete velocity field is element-wise solenoidal. Moreover, the sum of the velocity error times the viscosity plus the pressure $L^2$-error is proportional to the sum of the respective best errors. All proportionality constants are bounded in terms of shape regularity and do not depend on the viscosity. For simply connected two-dimensional domains, the velocity and pressure can be computed separately. The modification only affects the right-hand side aka load vector. The cost for building the modified load vector is proportional to the cost for building the standard load vector. Some numerical experiments illustrate our theoretical results.
\end{abstract}
\maketitle
%
%
%
%
\section{Introduction}\label{S:intro}
The Crouzeix-Raviart discretization \cite{CroRav} is a well-established nonconforming finite element method for the Stokes equations. It consists of piecewise constant pressures and piecewise affine velocity fields that are continuous at the barycentres of inter-element faces and that vanish at the barycentres of element faces on the boundary. This discretization is inf-sup stable on general simplicial meshes and provides a first-order approximation of both the analytical velocity and the analytical pressure in a rather straight-forward way. Other remarkable properties are that the discrete velocity field is element-wise solenoidal aka locally conservative and that, on simply connected two-dimensional domains, the velocity and pressure can be computed separately (cf. \cite{Bren14}, \cite[\S VI.8]{BreFor} and \S \ref{S:pract}). Still, one issue is that the nonconformity leads to a consistency error that can be bounded only under regularity assumptions on the analytical solution or the load functional, cf. \cite{VeeZan17a}. Moreover, the velocity error depends on the pressure error, as pointed out by Linke \cite{Lin}. 

Combining the approaches proposed in \cite{Lin} and \cite{VeeZan17b}, we construct a modified Crouzeix-Raviart discretization, which is quasi-optimal and pressure-robust. More precisely, the error of the discrete velocity field in a broken $H^1$-norm is proportional to the error of the best approximation to the analytical velocity field (cf. Theorem~\ref{T:main}). Moreover, the error of the discrete pressure in the $L^2$-norm is bounded, up to a constant, by the error of the best approximation to the analytical pressure plus the velocity error times the viscosity (cf. Theorem \ref{T:presserr}). All involved constants are independent of the viscosity and bounded in terms of the shape parameter of the underlying mesh. Thus, the error of the discrete velocity field is independent of the pressure, similarly to \cite{Lin} and unlike the standard discretization of Crouzeix and Raviart. Moreover, our estimates do not involve additional regularity, in contrast to the ones proved by Linke et al. \cite{Lin,LinMerNeiNeu18}.

The standard and the new discretization only differ in the right-hand side aka load vector. Consequently, the discrete velocity field is element-wise solenoidal also in the modified discretization and can be computed separately from the discrete pressure on simply connected two-dimensional domains. The additional cost for computing the modified load vector is proportional to the cost for computing the standard load vector.

The main idea to construct the modified load vector can be described as follows. We employ the \textit{smoothing} operator of \cite{VeeZan17b,VeeZan17c} to map Crouzeix-Raviart vector fields into piecewise polynomial, globally continuous vector fields before applying the analytical load functional. Since this operator does not guarantee pressure-robustness, we correct it by a locally computable stable right inverse of the divergence. More precisely, we solve a discrete Stokes problem with Scott-Vogelius elements \cite{GuzNei17,Qin,ScoVoga,ScoVogb,XuZha,Zha05} on a barycentric refinement of each mesh element (cf. \eqref{E:lSt} and \eqref{E:DivCorr}). Using the controvariant Piola transformation, it is actually sufficient to solve a fixed number of such local problems on a reference configuration (cf. \S \ref{S:pract}). The resulting operator has the additional property that element-wise solenoidal vector fields are mapped into exactly solenoidal vector fields. As already observed in \cite{Lin}, this is decisive to achieve pressure-robustness. 

The remainder of this article is organized as follows. In \S \ref{S:nonconf} we briefly recall the abstract setting and the main result of \cite{VeeZan17b}. In \S\S \ref{S:std} and \ref{S:mod} we then present the standard Crouzeix-Raviart discretization and our modification. In \S \ref{S:pract} we discuss the realization and the additional costs of our modification, as well as the possibility to decouple the computation of velocity and pressure. Finally, in \S\ref{S:numer} we illustrate and complement our abstract results by means of some numerical experiments in dimension two. 
%
%
\section{Quasi-optimal nonconforming methods}\label{S:nonconf}
For completeness and a better understanding, we outline the strategy of \cite{VeeZan17b} to design quasi-optimal nonconforming methods for symmetric elliptic problems, in a form adapted to our needs.

Consider a Hilbert space $V$ with scalar product $a(\cdot,\cdot)$. Given any continuous linear functional $\ell$ on $V$, we are looking for the unique function $u \in V$ such that
\begin{equation}\label{E:varpb}
a(u,v) = \ell(v)
\end{equation}
holds for all $v \in V$.

For the discretization of problem \eqref{E:varpb}, we consider a finite-dimensional space $V_\CT$. We assume that $a$ extends to a scalar product $a_\CT(\cdot,\cdot)$ on $V+V_\CT$ and denote by $\norm[\CT]{\cdot} = a_\CT(\cdot,\cdot)^\halb$ the induced energy norm. We also replace the load $\ell$ of \eqref{E:varpb} by a linear functional $\ell_\CT$ on $V_\CT$. Thus, we look for the unique function $u_\CT \in V_\CT$ such that
\begin{equation}\label{E:classnonconf}
a_\CT(u_\CT,v_\CT) = \ell_\CT(v_\CT)
\end{equation}
holds for all $v_\CT \in V_\CT$. We say that this is a (possibly) \textit{nonconforming} discretization of \eqref{E:varpb}, because $V_\CT$ is not required to be a subspace of $V$.

In standard nonconforming discretizations, like the one in \S\ref{S:std} below, the definition of $\ell_\CT$ often requires that $\ell$ in \eqref{E:varpb} has some extra-regularity. Hence, we cannot extend it to all continuous functionals on $V$. This generates a consistency error that cannot be bounded only in terms of the dual norm of $\ell$, cf. \cite[Remark~4.9]{VeeZan17a}. Consequently, the error $\norm[\CT]{u-u_\CT}$ is not proportional to the best approximation error $\inf_{v_\CT \in V_\CT} \norm[\CT]{u - v_\CT}$.

To overcome this drawback, Veeser and Zanotti suggest to consider a linear operator $E : V_\CT \longrightarrow V$ and look for the function $u_\CT \in V_\CT$ such that
\begin{equation}\label{E:quasioptnonconf}
a_\CT(u_\CT,v_\CT) = \ell(E v_\CT)
\end{equation}
holds for all $v_\CT \in V_\CT$. In this context, $E$ is also called \textit{smoothing} operator, as the nonconformity of $V_\CT$ often arises from a lack of smoothness. Notice that the discrete load $\ell_\CT$ in \eqref{E:quasioptnonconf} is well-defined for all continuous linear functionals on $V$. If $E$ also satisfies
\begin{equation}\label{E:consistency}
a_\CT(w_\CT,E v_\CT) = a_\CT(w_\CT,v_\CT) 
\end{equation}
for all $w_\CT,  v_\CT \in V_\CT$, then the solution of \eqref{E:quasioptnonconf} is quasi-optimal (in the norm $\norm[\CT]{\cdot}$), in the sense that
\begin{equation}\label{E:quasiopt}
\norm[\CT]{u - u_\CT} \le c_{\mathrm{qopt}} \inf_{v_\CT \in V_\CT} \norm[\CT]{u - v_\CT}.
\end{equation}
Furthermore, the best value of the constant $c_{\mathrm{qopt}}$ is the operator norm of $E$, see \cite[Corollary~2.7]{VeeZan17b}.

\begin{BEM}[Quasi-optimality and related notions]\label{R:qopt}
Quasi-optimality extends the well-known C\'{e}a's lemma and implies that the error $\norm[\CT]{u-u_\CT}$ of \eqref{E:quasioptnonconf} is equivalent to the best error $\inf_{v_\CT \in V_\CT} \norm[\CT]{u-v_\CT}$, in view of the inclusion $u_\CT \in V_\CT$, see e.g. \cite[Section~2.8]{BreSco} or \cite{VeeZan17a}. We mention that other authors call \textit{quasi-optimal} estimates in the form $\norm[\CT]{u-u_\CT} \leq c_{\mathrm{qopt}} \inf_{v_\CT \in V_\CT}\left( \norm[\CT]{u-v_\CT} + \mathrm{AG}_\CT(u-v_\CT) \right) $, where the norm $\norm[\CT]{\cdot}$ is augmented with $\mathrm{AG}_\CT(\cdot)$. The augmentation typically involves additional regularity of the solution $u$ or the load $\ell$ and is of higher-order, see  Carstensen et al. \cite{CarGalNat,CarSch} and Linke et al. \cite{LinMerNeiNeu18}. Such estimates are weaker than  \eqref{E:quasiopt}, in that the additional regularity required by $\mathrm{AG}_\CT(\cdot)$ obstructs a further bound of the right-hand side solely in terms of the best error $\inf_{v_\CT \in V_\CT} \norm[\CT]{u-v_\CT}$. 
\end{BEM}

%
%
\section{The standard Crouzeix-Raviart discretization}\label{S:std}
In what follows, $\Om$ is a connected bounded polyhedron in $\rz^d$, $d \geq 2$, with Lipschitz-continuous boundary. We denote by $\norm[]{\cdot}$ the $L^2$-norm on $\Om$. A subscript to $\norm[]{\cdot}$ indicates that we consider the $L^2$-norm only on the set specified by the subscript. Let $\CT$ be a shape-regular, face-to-face simplicial partition of $\Om$. We denote by $\CF$ and $\CV$ the faces and vertices, respectively, of the elements in $\CT$. A subscript to $\CF$ or $\CV$ indicates that we consider only those faces or vertices that are contained in the set specified by the subscript. We denote by $\betrag{K}$ and $\betrag{F}$ the $d$-dimensional Lebesgue measure of an element $K$ and the $(d - 1)$-dimensional Hausdorff measure of a face $F$, respectively. We write $h_K$ for the diameter of an element $K$. We associate with $\CT$ a so-called barycentric refinement $\CM$, which is obtained by connecting the vertices and the barycentre of every element $K \in\CT$, see \cite{GuzNei17}. $\CM_K$ stands for the restriction of $\CM$ to an element $K \in \CT$ and consists of $d + 1$ simplices. We denote by $c$ a generic nondecreasing function of the shape parameter of $\CT$. Such function does not need to be the same at different occurrences. 

For any integer $k \ge 0$, we denote by $\pz_k$ the space of polynomials of degree at most $k$ and set $S^{k,-1}(\CT) = \{ \vfi \in L^2(\Om): \forall K \in \CT \;\; \vfi_{|K} \in \pz_k \}$, where $\vfi_{|K}$ denotes the restriction of $\vfi$ to $K$. If $k \ge 1$ we set $S^{k,0}(\CT) = S^{k,-1}(\CT) \cap C(\Om) \subset H^1(\Om)$. The spaces $S^{k,-1}(\CM_K)$ and $S^{k,0}(\CM_K)$ are defined similarly with $\CT$ replaced by $\CM_K$. The (vector-valued) Crouzeix-Raviart space $CR(\CT)$ consists of all vector fields in $S^{1,-1}(\CT)^d$ that are continuous at the barycentres of interior faces and that vanish at the barycentres of boundary faces. Note that $CR(\CT) \not\subset H^1_0(\Om)^d$ due to the missing global continuity and the violation of the boundary condition. 

Denoting by $L^2_0(\Om)$ the space of all $L^2$-functions with mean value zero on $\Om$ and by $:$ and $\cdot$ the inner products of matrices and vectors respectively, the standard variational formulation of the Stokes equations with viscosity $\nu$ and load $\Vf \in L^2(\Om)^d$ consists in finding $\Vu \in H^1_0(\Om)^d$ and $p \in L^2_0(\Om)$ such that
\begin{equation}\label{E:Stokes}
\begin{aligned}
\nu \int_\Om \nabla \Vu : \nabla \Vv - \int_\Om p \Div \Vv &= \int_\Om \Vf \cdot \Vv &\forall \Vv \in H^1_0(\Om)^d \\
\int_\Om q \Div \Vu &= 0 &\forall q \in L^2_0(\Om).
\end{aligned}
\end{equation}
Similarly, denoting by $\gradT$ and $\divT$ the element-wise gradient and divergence, respectively, the standard Crouzeix-Raviart discretization of problem \eqref{E:Stokes} consists in finding $\Vu_\CT \in CR(\CT)$ and $p_\CT \in S^{0,-1}(\CT) \cap L^2_0(\Om)$ such that
\begin{equation}\label{E:CR}
\begin{split}
\nu \int_\Om \gradT \Vu_\CT : \gradT \Vv_\CT - \int_\Om p_\CT \divT \Vv_\CT &= \int_\Om \Vf \cdot \Vv_\CT \quad\quad\quad\,\forall \Vv_\CT \in CR(\CT) \\
\int_\Om q_\CT \divT \Vu_\CT &= 0 \quad\quad\forall q_\CT \in S^{0,-1}(\CT) \cap L^2_0(\Om).
\end{split}
\end{equation}

Problems \eqref{E:Stokes} and \eqref{E:CR} each admit a unique solution, see \cite[Example II.1.1, Example VI.3.10]{BreFor} or \cite{CroRav}. The discrete velocity field $\Vu_\CT$ is element-wise solenoidal, i.e. $\divT \Vu_\CT = 0$. If the analytical velocity field $\Vu$ is in $H^2(\Om)^d$ and the analytical pressure $p$ is in $H^1(\Om)$, then the error of the velocity measured in the broken $H^1$-norm $\norm[]{\nabla_\CT\cdot}$ and the error of the pressure measured in the $L^2$-norm decay linearly with respect to the mesh-size \cite{Bren14,CroRav}. In proving this estimate, one has to cope with the consistency error of the discretization due to the missing global continuity of the discrete velocity fields. Since, contrary to problem \eqref{E:Stokes}, problem \eqref{E:CR} is not defined for general $\Vf \in H^{-1}(\Om)^d$, it is not fully stable, meaning that its consistency error cannot be bounded in terms of the $H^{-1}$-norm of $\Vf$, cf. \cite{VeeZan17a} and Remark~\ref{R:no-smoothing} below. Moreover, the fact that element-wise solenoidal discrete velocities are in general not exactly solenoidal entails that \eqref{E:CR} is not pressure-robust, i.e. the velocity $H^1$-error depends on the pressure $L^2$-error, see \cite{Lin}. Both issues are addressed in the next section.

Problems \eqref{E:Stokes} and \eqref{E:CR} do not fit into the framework of \S \ref{S:nonconf}, since they are in saddle-point form. Yet, testing the first equation of \eqref{E:Stokes} with divergence-free functions, we obtain a reduced problem for the analytical velocity field, which fits into \eqref{E:varpb} with space, scalar product and load functional 
\begin{equation}\label{E:setting-Stokes}
\begin{gathered}
V = \{ \Vu \in H^1_0(\Om)^d : \Div \Vu = 0 \}\\
a(\Vu,\Vv) = \nu \int_\Om \nabla \Vu : \nabla \Vv
\qquad \text{and} \qquad
\ell(\Vv) = \int_\Om \Vf \cdot \Vv.
\end{gathered}
\end{equation}
Similarly, one can reduce \eqref{E:CR} to a problem for the discrete velocity field alone, which fits into \eqref{E:classnonconf} with 
\begin{equation}\label{E:setting-CR}
\begin{gathered}
V_\CT = \{ \Vu_\CT \in CR(\CT) : \divT \Vu_\CT = 0 \}\\
a_\CT(\Vu_\CT,\Vv_\CT) = \nu \int_\Om \gradT \Vu_\CT : \gradT \Vv_\CT
\qquad \text{and} \qquad
\ell_\CT(\Vv_\CT) = \int_\Om \Vf \cdot \Vv_\CT.
\end{gathered}
\end{equation}
Notice that $V_\CT \not\subset V$ and that the functional $\ell$ can be extended to $\Vf \in H^{-1}(\Om)^d$ but $\ell_\CT$ not. The stiffness matrix of problem \eqref{E:classnonconf} with \eqref{E:setting-CR} is symmetric positive definite but its condition number grows like $\goh(h^{-4})$ if $\CT$ is quasi-uniform. If $\Om$ is a simply connected two-dimensional domain, there is a basis of $V_\CT$ consisting of locally supported vector fields and the discrete pressure can be computed from the discrete velocity by sweeping through elements (cf. \S \ref{S:pract} and Algorithm \ref{A:presscomp}).
%
%
\section{The modified Crouzeix-Raviart discretization}\label{S:mod}
We now propose a modified version of \eqref{E:CR} and prove our main results.

\subsection{Construction of the discretization}
Motivated by the abstract framework of \S\ref{S:nonconf} and by the discussion in \S\ref{S:std}, we modify the standard Crouzeix-Raviart discretization \eqref{E:CR} as follows:  Find $\Vu_\CT \in CR(\CT)$ and $p_\CT \in S^{0,-1}(\CT) \cap L^2_0(\Om)$ such that
\begin{equation}\label{E:CRmod}
\begin{split}
\nu \int_\Om \gradT \Vu_\CT : \gradT \Vv_\CT - \int_\Om p_\CT \divT \Vv_\CT &= \int_\Om \Vf \cdot E \Vv_\CT \quad\quad\;\forall \Vv_\CT \in CR(\CT) \\
\int_\Om q_\CT \divT \Vu_\CT &= 0 \quad\quad\forall q_\CT \in S^{0,-1}(\CT) \cap L^2_0(\Om)
\end{split}
\end{equation}
where $E : CR(\CT) \longrightarrow H^1_0(\Om)^d$ is a linear operator (i.e. a smoothing operator, in the terminology of \S\ref{S:nonconf}). 

Since \eqref{E:CR} and \eqref{E:CRmod} only differ in the right-hand side, \cite[Example VI.3.10]{BreFor} or \cite{CroRav} again imply that problem \eqref{E:CRmod} admits a unique solution. In particular, the discrete velocity field $\Vu_\CT$ is element-wise solenoidal, showing that $\Vu_\CT \in V_\CT$. Testing the first equation with functions from $V_\CT$, we derive a reduced problem for $\Vu_\CT$, that fits into \eqref{E:classnonconf} and \eqref{E:quasioptnonconf} with 
\begin{equation}\label{E:setting-CR-mod}
\begin{gathered}
V_\CT = \{ \Vu_\CT \in CR(\CT) : \divT \Vu_\CT = 0 \}\\
a_\CT(\Vu_\CT,\Vv_\CT) = \nu \int_\Om \gradT \Vu_\CT : \gradT \Vv_\CT
\qquad \text{and} \qquad
\ell_\CT(\Vv_\CT) = \int_\Om \Vf \cdot E \Vv_\CT.
\end{gathered}
\end{equation}
In contrast to \eqref{E:setting-CR}, here $\ell_\CT$ can be extended to $\Vf \in H^{-1}(\Om)^d$.

We aim at constructing the operator $E$ so that the error of the discrete velocity field in the broken $H^1$-norm is proportional to the best approximation error to the analytical velocity. For this purpose, we preliminarily observe that \eqref{E:setting-CR-mod} is a nonconforming discretization of \eqref{E:setting-Stokes}. Hence, the results recalled in \S\ref{S:nonconf} indicate that $E$ should map $V_\CT$ into $V$ and satisfy \eqref{E:consistency}. Such conditions are sufficient to achieve \eqref{E:quasiopt}, with the best constant $c_\mathrm{qopt}$ given by the operator norm of $E$. For this reason, we also require that $E$ is $H^1$-stable, i.e. $\norm[]{\nabla E \Vv_\CT} \leq c \norm[]{\nabla_\CT \Vv_\CT}$ for all $\Vv_\CT \in CR(\CT)$.

Since the midpoint-rule is exact for affine functions, we have, for all $F \in \CF_\Om$ and $\Vv_\CT \in CR(\CT)$,
\begin{equation*}
\int_F \Vv_{\CT|K_{F1}} = \betrag{F} \Vv_{\CT|K_{F1}}(m_F) = \betrag{F} \Vv_{\CT|K_{F2}}(m_F) = \int_F \Vv_{\CT|K_{F2}},
\end{equation*}
where $K_{F1}$ and $K_{F2}$ are the two elements sharing the face $F$ and $m_F$ is its barycentre. The same observation reveals $\int_F \Vv_\CT = 0$ for $F \in \CF_{\partial \Om}$. Therefore, the integral $\int_F \Vv_\CT$ is defined without ambiguity for all faces and vanishes for boundary faces. Integrating by parts element-wise, one can easily check that, for all $\Vw_\CT, \Vv_\CT \in  CR(\CT)$,
\begin{equation*}
\int_\Om \gradT \Vw_\CT \colon \gradT \Vv_\CT
=
\sum_{K \in \CT} \sum_{F \in \CF_K}
\int_F \nabla (\Vw_\CT)_{|K} \Vn_K \cdot \Vv_\CT
\end{equation*}
and
\begin{equation*}
\int_\Om \gradT \Vw_\CT \colon \gradT E \Vv_\CT
=
\sum_{K \in \CT} \sum_{F \in \CF_K}
\int_F \nabla( \Vw_\CT)_{|K} \Vn_K \cdot E \Vv_\CT
\end{equation*}
where $\Vn_K$ is the outward unit normal vector of $K$. Thus, since $\gradT \Vw_\CT$ is element-wise constant, a sufficient condition for \eqref{E:consistency} is 
\begin{equation}\label{E:CRconsistency}
\int_F E \Vv_\CT = \int_F \Vv_\CT 
\end{equation}
for all $ F \in \CF$ and $\Vv_\CT \in CR(\CT)$. This identity always holds on faces $F \in \CF_{\partial \Om}$, because of the homogeneous boundary conditions in $H^1_0(\Om)^d$ and $CR(\CT)$.

We first construct a vector version of the smoothing operator in \cite[\S3.2]{VeeZan17b}. For every interior vertex $z \in \CV_\Om$, we denote by $\la_z \in S^{1,0}(\CT)$ the conforming first-order nodal basis function associated with the evaluation at $z$. We also choose an element $K_z \in \CT$ such that $z \in K_z$ and keep it fixed in what follows. With this notation, we define a \emph{simplified averaging} operator $A : CR(\CT) \longrightarrow S^{1,0}(\CT)^d \cap H^1_0(\Om)^d$ by
\begin{equation*}
A \Vv_\CT = \sum_{z \in \CV_\Om} \la_z \,\Vv_{\CT|K_z}(z).
\end{equation*}
Note that $A\Vv_\CT$ is set to zero at the vertices on $\partial \Om$.

Next, we associate a face-bubble $\psi_F \in S^{d,0}(\CT)$ with every interior face $F \in \CF_\Om$ 
\begin{equation*}
\psi_F = \frac{(2d - 1)!}{(d - 1)! \betrag{F}} \prod_{z \in \CV_F} \la_z.
\end{equation*}
The function $\psi_F$ is normalized so that $\int_{F^\prime} \psi_F = \delta_{FF^\prime}$ for all $F^\prime \in \CF$. Then, we define a \emph{bubble} operator $B : CR(\CT) \longrightarrow S^{d,0}(\CT)^d \cap H^1_0(\Om)^d$ by
\begin{equation*}
B \Vv_\CT = \sum_{F \in \CF_\Om} \psi_F \int_F \Vv_\CT.
\end{equation*}

The normalization of the bubble functions $\psi_F$ guarantees that $B$ satisfies \eqref{E:CRconsistency}. Yet, this operator is not $H^1$-stable, cf. \cite[Remark~3.5]{VeeZan17b}. Thus, we combine $A$ and $B$ and define a smoothing operator $C : CR(\CT)\longrightarrow S^{d,0}(\CT)^d \cap H^1_0(\Om)^d$ by 
\begin{equation*}
C \Vv_\CT = A \Vv_\CT + B ( \Vv_\CT - A \Vv_\CT).
\end{equation*}
According to \cite[Proposition~3.3]{VeeZan17b}, this operator is $H^1$-stable. Moreover, rearranging terms and exploiting the definition of $B$, we see that 
\begin{equation}\label{E:Cconsistency}
\int_F C \Vv_\CT 
= 
\int_F B\Vv_\CT 
+
\underbrace{\int_F (A\Vv_\CT - BA \Vv_\CT)}_{= 0}
=
\int_F \Vv_\CT 
\end{equation} 
for all $F \in \CF$ and $\Vv_\CT \in CR(\CT)$. This confirms that $C$ satisfies \eqref{E:CRconsistency}. 

Due to the Gauss theorem, identity \eqref{E:Cconsistency} entails, in particular,
\begin{equation}\label{E:w-conserv}
\int_K \Div (C \Vv_\CT) 
= 
\int_K \divT \Vv_\CT
\end{equation}
for all $K \in \CT$ and $\Vv_\CT \in CR(\CT)$. Unfortunately, this does not imply that $C$ maps $V_\CT$ into $V$, because $\Div (C\Vv_\CT) \in S^{d-1,-1}(\CT)$, cf. \cite[Remark~3.14]{VeeZan17c}. Recall the barycentric refinement $\CM$ of $\CT$. The crucial point of our modification of the Crouzeix-Raviart discretization is the construction of another $H^1$-stable smoothing operator $E : CR(\CT) \longrightarrow S^{d,0}(\CM)^d \cap H^1_0(\Om)^d$, correcting $C$, which preserves the validity of \eqref{E:Cconsistency} and additionally satisfies
\begin{equation}\label{E:s-conserv}
\Div (E \Vv_\CT)
=
\divT \Vv_\CT
\end{equation}
for all $\Vv_\CT \in CR(\CT)$. Indeed, this condition entails that $E$ maps $V_\CT$ into $V$. The desired correction is achieved by solving on every element $K \in \CT$ a discrete Stokes problem, based on stable Scott-Vogelius elements on $\CM_K$, with homogeneous momentum equation and suitable inhomogeneous continuity equation involving $\Div (C \Vv_\CT)$.

To make things precise, we consider, for every element $K \in \CT$ and every function $r \in L^2_0(K)$, the following discrete Stokes problem with unit viscosity, which consists in finding $\Vu_K \in S^{d,0}(\CM_K)^d \cap H^1_0(K)^d$ and $p_K \in S^{d-1,-1}(\CM_K) \cap L^2_0(K)$ such that
\begin{equation}\label{E:lSt}
\begin{aligned}
\int_K \nabla \Vu_K: \nabla \Vv - \int_K p_K \Div \Vv &= 0 &\forall \Vv \in S^{d,0}(\CM_K)^d \cap H^1_0(K)^d \\
\int_K q \Div \Vu_K &= \int_K r q &\forall q \in S^{d-1,-1}(\CM_K) \cap L^2_0(K).
\end{aligned}
\end{equation}

Denote by $S_K$ the mapping $r \mapsto \Vu_K$. As a consequence of \cite[Theorem 3.1]{GuzNei17}, problem \eqref{E:lSt} is uniquely solvable (entailing that $S_K$ is well-defined) and we have
\begin{equation}
\label{E:SK-stable}
\norm[K]{\nabla S_K r}
\leq
c
\norm[K]{r}.
\end{equation}
Assuming $r \in S^{d-1, -1}(\CM_K) \cap L^2_0(K)$ and extending both $r$ and $S_K r $ to zero outside $K$, we infer also
\begin{equation}
\label{E:SK-consistent}
(S_K r)_{|F} = 0
\qquad \text{and} \qquad
\Div(S_K r) = r
\end{equation}
for all $F \in \CF$. In particular, the latter property entails that the sum $\sum_{K \in \CT} S_K$ is a stable global right inverse of the divergence, which is defined on a subspace of $S^{d-1, -1}(\CM) \cap L^2_0(\Om)$ and can be computed locally.

In view of \eqref{E:w-conserv}, the restriction of $\Div (C\Vv_\CT) - \divT(\Vv_\CT)$ to each element $K \in \CT$ is in $S^{d-1,-1}(\CM_K) \cap L^2_0(K)$. Therefore, we define the announced smoothing operator $E : CR(\CT) \longrightarrow S^{d,0}(\CM)^d \cap H^1_0(\Om)^d$ by
\begin{equation}\label{E:DivCorr}
E \Vv_\CT = C \Vv_\CT - \sum_{K \in \CT} S_K \left( \Div (C \Vv_\CT) - \divT \Vv_\CT \right)
\end{equation}
for every $\Vv_\CT \in CR(\CT)$. The previously observed $H^1$-stability of $C$ and \eqref{E:SK-stable} ensure that $E$ is $H^1$-stable and its norm is bounded only in terms of the shape parameter of $\CT$. Furthermore, condition \eqref{E:CRconsistency} is fulfilled thanks to \eqref{E:Cconsistency} and the first part of \eqref{E:SK-consistent}, while \eqref{E:s-conserv} follows from the second part of \eqref{E:SK-consistent}.

\begin{BEM}[Divergence-free pairs]\label{R:DivFree}
The Scott-Vogelius pair employed in \eqref{E:lSt} is divergence-free, in the sense that the divergence maps the discrete velocity space onto the discrete pressure space. This property is needed to compute the local right inverses $S_K$ of the divergence, which are then used to correct the operator $C$. An alternative construction, based on a different divergence-free pair, can be found in \cite{LinMerNeiNeu18}. The analysis of the Scott-Vogelius pair was initiated in \cite{ScoVoga,ScoVogb}, where it was pointed out that stability can be obtained only for certain combinations of the mesh geometry and the polynomial degree $k$. The stability on the barycentric refinement of a simplicial mesh was proved by Qin \cite{Qin} in 2D for $k \geq 2$ and by Zhang \cite{Zha05} in 3D for $k \geq 3$. The recent paper \cite{GuzNei17} by Guzm\'{a}n and Neilan generalizes these results in $\rz^d$ for $k \geq d$.  
\end{BEM}

\subsection{Error estimates}\label{SS:velocity-error}
Consider the modified Crouzeix-Raviart discretization \eqref{E:CRmod} with the operator $E$ from \eqref{E:DivCorr}. The above-mentioned properties of $E$ imply our first main result.

\begin{THM}[Quasi-optimal velocity error]\label{T:main}
Denote by $(\Vu, p)$ and $(\Vu_\CT, p_\CT)$ the unique solutions of problems \eqref{E:Stokes} and \eqref{E:CRmod}, respectively. There is a constant $\shps_1$, which only depends on the shape parameter of $\CT$ and not on the viscosity $\nu$, such that
\begin{equation}\label{velocity-error-estimate}
\norm{\nabla \Vu - \gradT \Vu_\CT} \le \shps_1 \inf_{\Vv_\CT \in V_\CT} \norm{\nabla \Vu - \gradT \Vv_\CT} = \shps_1 \inf_{\Vv_\CT \in CR(\CT)}\norm{\nabla \Vu - \gradT \Vv_\CT}.
\end{equation}
\end{THM}

\begin{proof}
The first inequality follows from identities \eqref{E:CRconsistency} and \eqref{E:s-conserv} and the $H^1$-stability of $E$, in view of the abstract discussion in section 2 or \cite[Corollary~2.7]{VeeZan17b}. The identity of the best errors can be checked with the help of the so-called Crouzeix-Raviart quasi-interpolation operator $I_\CT: H^1_0(\Om)^d \to CR(\CT)$, which is defined by $\int_F I_\CT \Vu = \int_F \Vu$ for all faces $F \in \CF$. Integrating by parts element-wise, we see that $\int_K \gradT (I_\CT \Vu) = \int_K \nabla \Vu$ and $\int_K \divT (I_\CT \Vu) = \int_K \Div \Vu$ for all $K \in \CT$. The latter identity entails $I_\CT \Vu \in V_\CT$, because $\Vu \in V$ and $\divT(I_\CT \Vu)$ is element-wise constant. Combining this fact with the first identity implies $\inf_{\Vv_\CT \in V_\CT} \norm{\nabla \Vu - \gradT \Vv_\CT} = \norm{\nabla \Vu - \gradT (I_\CT\Vu)} = \inf_{\Vv_\CT \in CR(\CT)}\norm{\nabla \Vu - \gradT \Vv_\CT}$.
\end{proof}

Theorem~\ref{T:main} combines the advantages of \cite{Lin,LinMerNeiNeu18} and \cite{VeeZan17b}, because the velocity $H^1$-error is independent of the pressure and proportional to the error of the best approximation to the analytical velocity. This implies that the velocity error of the standard discretization \eqref{E:CR} can be smaller, but not arbitrarily smaller, than the error of the modified one. The constant $c_1$ quantifies the maximum gap, which is uniformly bounded for shape regular meshes. Conversely, the next remark shows that \eqref{E:CRmod} can significantly outperform \eqref{E:CR} for certain loads. 

\begin{BEM}[Nonconforming discretizations without smoothing] \label{R:no-smoothing}
The error estimate of Theorem~\ref{T:main}, combined with the triangle inequality, reveals that the solution $\Vu_\CT$ of \eqref{E:CRmod} depends continuously on the analytical velocity $\Vu$ in the broken $H^1$-norm. This property hinges on the use of the smoothing operator $E$, which maps the Crouzeix-Raviart test functions into $H^1_0(\Om)^d$, before the application of the load functional. In particular, $\Vu_\CT$ is well-defined for all loads $\Vf\in H^{-1}(\Om)^d$. In contrast, the standard Crouzeix-Raviart discretization \eqref{E:CR} is only defined under the regularity assumption 
\begin{equation*}
\norm[CR^*]{\Vf} 
:= 
\sup_{\Vv_\CT \in CR(\CT) \setminus \{0\}}
\dfrac{\int_\Om \Vf \cdot \Vv_\CT}{\norm[]{\gradT \Vv_\CT} }
< \infty.
\end{equation*} 	
Since $\norm[CR^*]{\cdot}$ is not equivalent to the $H^{-1}$-norm, the discrete velocity $\Vu_\CT$ of  \eqref{E:CR} does not depend continuously on $\Vu$ in the broken $H^1$-norm. Therefore, the velocity $H^1$-error of the standard discretization can be arbitrarily larger than the best error $\inf_{\Vv_\CT \in CR(\CT)} \norm[]{\gradT(\Vu-\Vv_\CT)}$, provided $\norm[CR^*]{\Vf}$ is sufficiently larger than the $H^{-1}$-norm of $\Vf$. This observation is not specific to \eqref{E:CR} and concerns any other nonconforming discretization without smoothing, including the ones in \cite{Lin,LinMerNeiNeu18}.   
\end{BEM}

The following remarks shed additional light on the error estimate of Theorem~\ref{T:main}, in connection with the results available for other existing discretizations of the Stokes equations.

\begin{BEM}[Quasi-optimality and divergence-free pairs]\label{R:qopt-DivFree}
The quasi-optimal estimate of Theorem~\ref{T:main} distinguishes \eqref{E:CRmod} from nonconforming methods without smoothing but also from several other conforming finite element methods for \eqref{E:Stokes}. Indeed, to our best knowledge, estimates in this form have been previously obtained only with conforming and divergence-free pairs, like the one of Scott and Vogelius, see Remark~\ref{R:DivFree}. Moreover, if we restrict the attention to first-order discretizations, only few such pairs are known to be stable, either requiring special mesh geometries or the use of non-polynomial basis functions, cf. \cite{GuzNei14,GuzNei17,Zha08}. Theorem~\ref{T:main} shows that quasi-optimality is actually not restricted to conforming and divergence-free pairs but can be achieved by means of a proper discretization of the load functional.  
\end{BEM}

\begin{BEM}[Shape regularity and anisotropic meshes]\label{R:shape-regularity}
Acosta and Dur\'{a}n considered in \cite{AcoDur} possibly anisotropic meshes, fulfilling maximum angle conditions, in dimension two and three. For certain smooth solutions of \eqref{E:Stokes}, they proved that the velocity error of \eqref{E:CR} in the broken $H^1$-norm converges with maximum decay rate. A counterpart of this result cannot be derived for \eqref{E:CRmod} by Theorem~\ref{T:main}. In fact, the constant $\shps_1$ in \eqref{velocity-error-estimate} depends on the shape parameter of $\CT$ and it does not seem possible to avoid such dependence if the operator $E$ is defined as in \eqref{E:DivCorr}, cf. \S\ref{SS:smooth-sol-anis}. Still, even in case $\shps_1$ is indeed large, it is not obvious that \eqref{E:CR} always outperforms \eqref{E:CRmod}. This is illustrated by a numerical experiment in \S\ref{S:numer}.
\end{BEM}

The abstract framework of \S\ref{S:nonconf} cannot be used to bound the pressure $L^2$-error. For this purpose, we combine the previous bound of the velocity error, the properties of $E$ and standard techniques for saddle point problems. We first observe that integration by parts element-wise, together with \eqref{E:CRconsistency}, the first equations of \eqref{E:Stokes} and \eqref{E:CRmod} and the fact that $E$ maps into $H^1_0(\Om)^d$ imply
\begin{equation*}
\begin{split}
\nu \int_\Om \gradT \Vu_\CT : \nabla (E \Vv_\CT) - \int_\Om p_\CT \Div (E \Vv_\CT) &= \nu \int_\Om \gradT \Vu_\CT : \gradT \Vv_\CT - \int_\Om p_\CT \divT \Vv_\CT \\
&= \int_\Om \Vf \cdot E \Vv_\CT \\
&= \nu \int_\Om \nabla \Vu : \nabla (E \Vv_\CT) - \int_\Om p \Div (E \Vv_\CT)
\end{split}
\end{equation*}
for all $\Vv_\CT \in CR(\CT)$. Thus, we derive the identity
\begin{equation}\label{E:perr}
\int_\Om (p - p_\CT) \Div (E \Vv_\CT) = \nu \int_\Om (\nabla \Vu - \gradT \Vu_\CT) : \nabla (E \Vv_\CT).
\end{equation}
Next, we compare $p_\CT$ with the $L^2$-projection $\pi_\CT p \in S^{0,-1}(\CT)$ of $p$
\begin{equation}\label{E:tri}
\norm{p - p_\CT} \le \norm{p - \pi_\CT p} + \norm{\pi_\CT p - p_\CT} 
\end{equation}
and use \eqref{E:CRconsistency} and the Gauss theorem to obtain
\begin{equation}\label{E:pdiffE}
\int_\Om (\pi_\CT p - p_\CT) \divT \Vv_\CT = \int_\Om (\pi_\CT p - p_\CT) \Div (E \Vv_\CT) 
\end{equation}
for all $\Vv_\CT \in CR(\CT)$. Finally, the uniform stability of the Crouzeix-Raviart element \cite[Example VI.3.10]{BreFor}, \cite{CroRav} implies
\begin{equation}\label{E:pstab}
\be \norm{\pi_\CT p - p_\CT} \le \sup_{\Vv_\CT \in CR(\CT) \setminus \{ 0 \}} \frac{\int_\Om (\pi_\CT p - p_\CT) \divT \Vv_\CT}{\norm{\gradT \Vv_\CT}},
\end{equation}
where $\be = \inf_{q \in L^2_0(\Om) \setminus \{ 0 \}} \sup_{\Vv \in H^1_0(\Om)^d \setminus \{ 0 \}} \frac{\int_\Om q \Div \Vv}{\norm{q} \norm{\nabla \Vv}}$ is the inf-sup constant of the divergence operator. Recall that $\be$ does not depend on the viscosity $\nu$ and $\beta^{-1}$ is bounded in terms of the ratio $\mathrm{diam}(\Om) /R$, provided $\Om$ is star-shaped with respect to a ball of radius $R$, cf. \cite{Bog}. Consequently, the estimate in Theorem~\ref{T:presserr} below depends on this ratio, while the one in Theorem~\ref{T:main} is independent of it. 

Combining \eqref{E:tri}, \eqref{E:pstab}, \eqref{E:pdiffE}, \eqref{E:perr}, the $H^1$-stability of $E$ and Theorem~\ref{T:main} proves our second main result.

\begin{THM}[Pressure error]\label{T:presserr}
Denote by $(\Vu, p)$ and $(\Vu_\CT, p_\CT)$ the unique solutions of problems \eqref{E:Stokes} and \eqref{E:CRmod}, respectively, and let $c_1$ be as in Theorem~\ref{T:main}. There are two constants $\shps_2$ and $\shps_3$, which only depend on the shape parameter of $\CT$ and the inf-sup constant $\be$ of the divergence operator but not on the viscosity $\nu$, such that
\begin{equation*}
\begin{split}
\norm{p - p_\CT} 
\le \shps_2 \inf_{q_\CT \in S^{0,-1}(\CT)}\norm{p - q_\CT} + \shps_1 \shps_3 \nu \inf_{\Vv_\CT \in CR(\CT)}\norm{\nabla \Vu - \gradT \Vv_\CT}.
\end{split}
\end{equation*}
\end{THM}

Estimates of the pressure error in this form are well-established for all conforming stable pairs and not only for the divergence-free ones, see \cite[Theorem~II.2.1]{BreFor}. Correspondingly, our proof exploits \eqref{E:CRconsistency} and the $H^1$-stability of $E$, but not \eqref{E:s-conserv}. In contrast to similar results in \cite{CroRav} and \cite{Lin,LinMerNeiNeu18}, Theorem~\ref{T:presserr} does not assume additional regularity of the solution.  

%
%
\section{Practical aspects}\label{S:pract}


\subsection{Assembly of the modified discretization}
\label{SS:assembly}

Assume that $\{ \Vv_\CT^1, \dots, \Vv_\CT^N \}$ and $\{ \Vw_\CM^1,\dots, \Vw_\CM^M \}$ are bases of $CR(\CT)$ and $S^{d,0}(\CM)^d \cap H^1_0(\Om)^d$, respectively. Problem \eqref{E:CRmod} requires the computation of $\int_\Om \Vf \cdot E \Vv_\CT^i$ for $i=1, \dots, N$. Let $\mathbb{E} \in \rz^{N\times M}$ and $\overrightarrow{\Vf} \in \rz^M$ be defined as follows
\begin{equation*}
E \Vv_\CT^i = \sum_{j=1}^M \mathbb{E}_{ij} \Vw_\CM^j 
\qquad \text{and} \qquad 
(\overrightarrow{\Vf})_j = \int_\Om \Vf \cdot \Vw_\CM^j 
\end{equation*}
for $i=1,\dots, N$ and $j=1, \dots, M$. Few algebraic manipulations reveal
\begin{equation*}
\int_\Om \Vf \cdot E \Vv_\CT^i = (\mathbb{E} \overrightarrow{\Vf} )_i.
\end{equation*}
We aim at showing that both $\mathbb{E}$ and $\overrightarrow{\Vf}$ can be computed with $\goh(N)$ operations, if one uses standard nodal bases of $CR(\CT)$ and $S^{d,0}(\CM)^d \cap H^1_0(\Om)^d$. 

The nodal basis $\{ \vfi_F \Ve_1, \ldots, \vfi_F \Ve_d : F \in \CF_\Om \}$  of $CR(\CT)$ is known to be a convenient choice for assembling the left-hand side of \eqref{E:CR} and \eqref{E:CRmod}. Here  $\{ \Ve_1, \ldots, \Ve_d \}$ stands for the standard basis of $\rz^d$ and each $\vfi_F \in S^{1,-1}(\CT)$ is uniquely defined by the following properties: $\vfi_F$ is continuous at barycentres of interior faces and is normalized so that $\vfi_F(m_{F^\prime}) = \betrag[]{F}^{-1}\delta_{FF^\prime}$ for all $F^\prime \in \CF$. Note that the midpoint quadrature rule implies $\int_{F^\prime} \vfi_F = \delta_{FF^\prime}$ for all $F^\prime \in \CF$. Similarly, we consider the basis $\{ \vfi_z \Ve_1, \ldots, \vfi_z \Ve_d : z \in \CV_\Om^d(\CM) \}$ of $S^{d,0}(\CM)^d \cap H^1_0(\Om)^d$, where $\CV_\Om^d(\CM)$ denotes the interior Lagrange nodes of degree $d$ of $\CM$ and $\vfi_z$ is the nodal basis function of $S^{d,0}(\CM)$ associated with the evaluation at $z$. 
 
With the latter basis, the computation of $\overrightarrow{\Vf}$ requires $\goh(M) = \goh(N)$ operations. Moreover, each entry of $\mathbb{E}$ can be computed in $\goh(1)$ operations and the maximum number of nonzero entries in each row of $\mathbb{E}$ is bounded by a constant depending on the shape parameter of $\CT$. This entails that the cost for computing $\mathbb{E}$ is $\goh(N)$ operations. Since the left-hand side of \eqref{E:CR} is the same as in \eqref{E:CRmod}, we infer that the cost for assembling the modified discretization is the same as the one for assembling the standard one. 


To check our claim concerning the matrix $\mathbb{E}$, we first observe that each basis function $\vfi_F \Ve_k$, with $F \in \CF_\Om$ and $k \in \{ 1, \ldots, d \}$, is supported in the union $\omega_F$ of the two elements sharing $F$. According to \cite[\S3.2]{VeeZan17b}, the function $C(\vfi_F \Ve_k)$ is supported in the union of all elements touching $\omega_F$ and can be computed with $\goh(1)$ operations. Hence, to obtain $E(\vfi_F \Ve_k)$ from $C(\vfi_F \Ve_k)$, one has to solve the local Stokes problem \eqref{E:lSt} with load $r = \Div (C (\vfi_F \Ve_k)) -  \divT (\vfi_F \Ve_k)$ on each element $K \in \CT$ such that $K \cap \omega_F \neq \emptyset$. For this purpose, an efficient strategy is to precompute the solution of the local problem on a reference element  $\widehat{K}$ with load $\widehat{r}$, where $\widehat{r}$ varies in a basis $\{\widehat{r}_1, \ldots, \widehat{r}_m\}$ of $\pz_{d-1} \cap L^2_0(\widehat{K})$. Then, the solution of \eqref{E:lSt} can be simply obtained by means of the controvariant Piola's transformation, see \cite[\S III.1.3]{BreFor}. This confirms that $E(\vfi_F \Ve_k)$ is supported in $\{ K \in \CT : K \cap \om_F \neq \emptyset \}$ and can be computed from $\vfi_F \Ve_k$ with $\goh(1)$ operations. We conclude recalling the definition of $\mathbb{E}$ and our choice of the nodal basis of $S^{d,0}(\CM)^d \cap H^1_0(\Om)^d$. 

\subsection{Solution of the modified discretization} 
\label{SS:solution}
As for the standard Crouzeix-Raviart discretization \eqref{E:CR}, the computation of the velocity and pressure solving the modified problem \eqref{E:CRmod} can be decoupled, if $\Om$ is a simply connected two-dimensional domain. This procedure is essentially known in the literature. We show how to adapt it to the modified problem, for the sake of completeness. An extension to three dimensions is given in \cite{Hec}.

First recall that the discrete velocity $\Vu_\CT$ solving \eqref{E:CRmod} is in the subspace $V_\CT$, consisting of element-wise solenoidal Crouzeix-Raviart functions, cf. \eqref{E:setting-CR-mod}. For a simply connected two-dimensional domain, a basis of $V_\CT$ is given by the union of the sets $\{ \vfi_F \Vt_F : F \in \CF_\Om \}$ and $\{ \Vw_z : z \in \CV_\Om \}$, see \cite[Example VI.8.1, Figure~VI.35]{BreFor}. Here, $\Vt_F$ is a unit tangent vector to $F$ and $\Vw_z$ is a \emph{vortex} around $z$, defined by
\begin{equation*}
\Vw_z = \sum_{F_z \in \CF_z}  \vfi_{F_z} \Vn_{F_z}
\end{equation*}
where $\CF_z$ consists of all edges meeting at $z$ and $\Vn_{F_z}$ is a unit normal vector to $F_z$, oriented counterclockwise with respect to $z$. This basis can be used to compute $\Vu_\CT$ from problem \eqref{E:quasioptnonconf} with \eqref{E:setting-CR-mod}.

Next, let $\Vn_F$, $F \in \CF_\Om$, be a unit normal vector to $F$ and denote by $\jump[F]{\cdot}$ the jump across $F$ in direction $\Vn_F$. Assuming that $\Vu_\CT$ is known, we test the first equation of \eqref{E:CRmod} with $\Vv_\CT = \vfi_F \Vn_F$. Rearranging terms, the Gauss theorem yields
\begin{equation}\label{E:pRedEq}
- \jump[F]{p_\CT} = \int_{\om_F} \Vf \cdot E (\vfi_F \Vn_F) - \int_{\om_F} \gradT \Vu_\CT : \gradT (\vfi_F \Vn_F)
\end{equation}
for all $F \in \CF_\Om$. This set of conditions determines the discrete pressure $p_\CT$ uniquely, because the first equation of \eqref{E:CRmod} is automatically fulfilled for all test functions $\Vv_\CT = \vfi_F \Vt_F$ with $F \in \CF_\Om$. 

Let $K_0, K \in \CT$ be connected by a path of triangles $(K_i)_{i=1}^n$ such that $K_n = K$ and $F_i := K_i \cap K_{i-1}$, $i=1, \dots, n$, is an interior edge of $\CT$. Since $p_\CT$ is element-wise constant, we have
\begin{equation}\label{E:pJumps}
p_{\CT|K} - p_{\CT|K_0}
=
\sum_{i=1}^n \Vn_{F_{i}} \cdot \Vn_{K_i} \jump[{F_i}]{p_\CT}
\end{equation} 
where $\Vn_{K_i}$ is the outward unit normal vector of $K_i$. This identity has two interesting consequences. First, the sum $\sum_{i=1}^n \Vn_{F_{i}} \cdot \Vn_{K_i} \jump[{F_i}]{p_\CT}$ vanishes for $K = K_0$, meaning that the choice of the path connecting two triangles is irrelevant. Second, comparing with \eqref{E:pRedEq}, we see that the value of $p_\CT$ in any triangle $K \neq K_0$ depends only on the load $\Vf$, the discrete velocity $\Vu_\CT$ and the value of $p_\CT$ in $K_0$. Therefore, we can compute $p_\CT$ as follows. Defining $\overline{p}_\CT := p_\CT - p_{\CT|K_0}$, we have $\overline{p}_{\CT|K_0} = 0$ and, for all $K \neq K_0$, 
\begin{equation*}
\overline{p}_{\CT|K}
=
\sum_{i=1}^n
\Vn_{F_{i}} \cdot \Vn_{K_i}
\left( 
\int_{\om_{F_i}} \gradT \Vu_\CT : \gradT (\vfi_{F_i} \Vn_{F_i})
-
\int_{\om_{F_i}} \Vf \cdot E (\vfi_{F_i} \Vn_{F_i})
\right) 
\end{equation*}
where $(K_i)_{i=1}^n$ is any path connecting $K_0$ and $K$. The constraint $\int_\Om p_\CT = 0$ further implies $\sum_{K\in \CT} \betrag[]{K} \overline{p}_{\CT|K} = -\betrag[]{\Om} p_{\CT|K_0} $, showing that $p_\CT$ can be easily recovered from $\overline{p}_\CT$. The following algorithm implements this procedure in a way that requires a number of operations proportional to $\dim(S^{0,-1}(\CT))$.

\begin{algorithm}[htp]
\caption{Pressure Computation}\label{A:presscomp}
\begin{algorithmic}[1]
\algrenewcommand\algorithmicensure{\textbf{Provide}:}
 \Require $\Vu_\CT$ solution of \eqref{E:quasioptnonconf} with \eqref{E:setting-CR-mod}
 \Ensure $p = p_\CT$ pressure solving \eqref{E:CRmod}
  \State $p \gets 0$
  \State choose an element $K_0 \in \CT$
  \State $M \gets K_0$, $\CU \gets \CT \setminus \{ K_0 \}$
   \While{$\CU \ne \emptyset$}
    \State choose $K \in \CU$ so that $F := K \cap M$ is an interior edge of $\CT$
    \State $\displaystyle p_{|K} \gets p_{| \om_F \cap M} + \Vn_F \cdot \Vn_K \left( \int_{\om_F} \gradT \Vu_\CT : \gradT (\vfi_F \Vn_F) - \int_{\om_F} \Vf \cdot E (\vfi_F \Vn_F)\right) $ 
    \State $M \gets M \cup K$, $\CU \gets \CU \setminus \{ K \}$
  \EndWhile
  \State $\displaystyle p \gets p - \sum_{K \in \CT} \frac{\betrag{K}}{\betrag{\Om}} p_{|K}$.
\end{algorithmic}
\end{algorithm}

%
%
\section{Numerical experiments}\label{S:numer}

In this section we report and discuss the results of four numerical experiments, that are intended to illustrate and partially complement Theorems~\ref{T:main} and \ref{T:presserr}. In particular, we compare the modified Crouzeix-Raviart discretization \eqref{E:CRmod}, with $E$ as in \eqref{E:DivCorr}, and the standard one \eqref{E:CR}, whenever the latter is defined. If an exact solution is available, we compute also the best approximation $H^1$-error to the analytical velocity and the best approximation $L^2$-error to the analytical pressure,
\begin{equation*}\label{E:best-errors}
e_\CT(\Vu) := \inf_{\Vv_\CT \in CR(\CT)}\norm{\nabla \Vu - \gradT \Vv_\CT}
\qquad \text{and} \qquad
e_\CT(p) := \inf_{q_\CT \in S^{0,-1}(\CT)}\norm{p - q_\CT}.
\end{equation*}
As mentioned in the proof of Theorem~\ref{T:main}, the former is given by
\begin{equation*}
e_\CT(\Vu) = \norm{\nabla \Vu - \gradT (I_\CT \Vu)}
\end{equation*}
where $I_\CT$ is the Crouzeix-Raviart quasi-interpolation operator.

All experiments have been implemented in ALBERTA 3.0 \cite{HeiKoeKriSchSie,SchSie} and concern the two-dimensional Stokes equations \eqref{E:Stokes}, posed in the unit square, with unit viscosity, i.e.
\begin{equation*}\label{E:numerics-setting}
d=2, \qquad \qquad
\Om = (0,1) \times (0,1), \qquad \qquad
\nu = 1.
\end{equation*}

\subsection{Smooth solution} \label{SS:smooth-sol}

We first consider a test case with smooth solution, namely
\begin{equation*}\label{E:test1-sol}
\Vu(x_1, x_2) = \Curl\left(\:x_1^2(x_1-1)^2x_2^2(x_2-1)^2\:\right),
\qquad
p(x_1, x_2) = (x_1 - 0.5)(x_2 - 0.5)
\end{equation*}
where $\Curl(v) := (\partial_{x_2} v, -\partial_{x_1} v)$. We solve \eqref{E:CR} and \eqref{E:CRmod} on the following sequence $(\CT_n)_{n\geq 0}$ of uniform meshes. We divide $\Om$ into $2^n \times 2^n$ squares, with edges parallel to the lines $x_2 = 0$ and $x_1 = 0$ and edge length $2^{-n}$. Then, we obtain $\CT_n$ by drawing, for each square, the diagonal parallel to the line $x_1 = x_2$, see Figure~\ref{F:meshes-ex1-2}.   
Since $(\Vu, p) \in H^2(\Om)^2 \times H^1(\Om)$, both $e_{\CT_n}(\Vu)$ and $e_{\CT_n}(p)$ converge to zero with maximum decay rate $2^{-n} \approx (\#\CT_n)^{-0.5}$. 

To assess the quality of the standard and the modified Crouzeix-Raviart discretizations, we compute the ratios
\begin{equation}\label{E:error/best}
\gamma_n(\Vu) := \dfrac{\norm[L^2(\Om)]{\nabla \Vu - \nabla_{\CT_n} \Vu_{\CT_n}}}{e_{\CT_n}(\Vu)} 
\qquad \text{and} \qquad
\gamma_n(p) := \dfrac{\norm[L^2(\Om)]{p - p_{\CT_n}}}{e_{\CT_n}(p)} 
\end{equation}
where $(\Vu_{\CT_n}, p_{\CT_n}) \in CR(\CT_n) \times S^{0,-1}(\CT_n)$ denotes either the solution of \eqref{E:CR} or the one of \eqref{E:CRmod}. Some values of $\gamma_n(\Vu)$ and $\gamma_n(p)$ are displayed in the first column of Tables~\ref{F:velerr/best} and \ref{F:preerr/best}, respectively. They indicate that the velocity $H^1$-error of the modified discretization is larger than the one of the standard discretization by a factor between $1.3$ and $1.5$. Instead, the corresponding pressure $L^2$-errors are nearly of the same size for sufficiently large $n$ and the error of the modified discretization is smaller for the first values of $n$. More generally, one can expect that both discretizations perform similarly for smooth solutions of \eqref{E:Stokes}, on shape regular sequences of meshes.  

\begin{figure}
	\hfill
	\subfloat{\includegraphics[width=0.4\hsize]{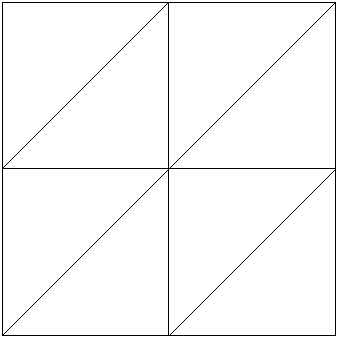}}
	\hfill
	\subfloat{\includegraphics[width=0.4\hsize]{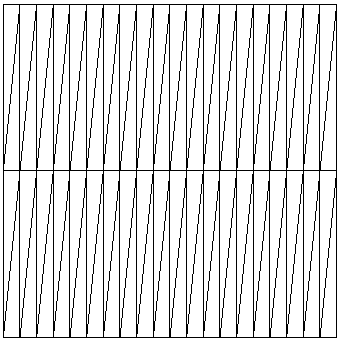}}
	\hfill
	\caption{Meshes $\CT_1$ from \S\ref{SS:smooth-sol} (left) and $\CT_1^{10}$ from \S\ref{SS:smooth-sol-anis} (right).}
	\label{F:meshes-ex1-2}
\end{figure}

\begin{table}
	\begin{tabular}{r|c|c|c|c}
		 &
		 &
		m = 10 &
		m = 20 &
		m = 40 \\
		n &
		\texttt{std}  \hspace{5pt} \texttt{mod} &
		\texttt{std}  \hspace{5pt} \texttt{mod} &
		\texttt{std}  \hspace{5pt} \texttt{mod} &
		\texttt{std}  \hspace{5pt} \texttt{mod} 
		\\[1ex]
		\hline
		&&&&
		\\[-1.5ex]
		2 & 1.37 \hspace{5pt} 2.07& 1.39 \hspace{5pt} 2.03 
		& 1.39 \hspace{5pt} 2.03 & 1.39 \hspace{5pt} 2.03
		\\
		3 & 1.48 \hspace{5pt} 2.06 & 1.50 \hspace{5pt} 2.04
		& 1.50 \hspace{5pt} 2.04 & 1.50 \hspace{5pt} 2.04
		\\
		4 & 1.54 \hspace{5pt} 2.05 & 1.55 \hspace{5pt} 2.05
		& 1.55 \hspace{5pt} 2.05 & 1.55 \hspace{5pt} 2.05
		\\
		5 & 1.57 \hspace{5pt} 2.05 & 1.57 \hspace{5pt} 2.05
		& 1.57 \hspace{5pt} 2.05 & 1.57 \hspace{5pt} 2.05
		\\
		6 & 1.58 \hspace{5pt} 2.05 & 1.58 \hspace{5pt} 2.05
		& 1.58 \hspace{5pt} 2.05 & 1.58 \hspace{5pt} 2.05 
	\end{tabular}
	\vspace{1ex}
	\caption{Ratios $\gamma_n(\Vu)$ from \S\ref{SS:smooth-sol} and $\gamma_n^m(\Vu)$, $m \in \{10, 20, 40\}$, from \S\ref{SS:smooth-sol-anis} for the standard (\texttt{std}) and the modified (\texttt{mod}) Crouzeix-Raviart discretizations.}
	\label{F:velerr/best}
	\begin{tabular}{r|c|c|c|c}
		&
		&
		m = 10 &
		m = 20 &
		m = 40 \\
		n &
		\texttt{std}  \hspace{5pt} \texttt{mod} &
		\texttt{std}  \hspace{5pt} \texttt{mod} &
		\texttt{std}  \hspace{5pt} \texttt{mod} &
		\texttt{std}  \hspace{5pt} \texttt{mod} 
		\\[1ex]
		\hline
		&&&&
		\\[-1.5ex]
		2 & 1.44 \hspace{5pt} 1.09 & 1.57 \hspace{5pt} 1.12 
		& 1.57 \hspace{5pt} 1.12 & 1.57 \hspace{5pt} 1.12
		\\
		3 & 1.41 \hspace{5pt} 1.10 & 1.41 \hspace{5pt} 1.11
		& 1.41 \hspace{5pt} 1.11 & 1.41 \hspace{5pt} 1.11
		\\
		4 & 1.25 \hspace{5pt} 1.07 & 1.22 \hspace{5pt} 1.07
		& 1.22 \hspace{5pt} 1.07 & 1.21 \hspace{5pt} 1.07
		\\
		5 & 1.14 \hspace{5pt} 1.06 & 1.11 \hspace{5pt} 1.06
		& 1.11 \hspace{5pt} 1.06 & 1.11 \hspace{5pt} 1.06
		\\
		6 & 1.08 \hspace{5pt} 1.06 & 1.09 \hspace{5pt} 1.07
		& 1.07 \hspace{5pt} 1.06 & 1.07 \hspace{5pt} 1.06		 
	\end{tabular}
	\vspace{1ex}
	\caption{Ratios $\gamma_n(p)$ from \S\ref{SS:smooth-sol} and $\gamma_n^m(p)$, $m \in \{10, 20, 40\}$, from \S\ref{SS:smooth-sol-anis} for the standard (\texttt{std}) and the modified (\texttt{mod}) Crouzeix-Raviart discretizations.}
	\label{F:preerr/best}
\end{table}

\subsection{Smooth solution and anisotropic meshes}\label{SS:smooth-sol-anis}

The purpose of this experiment is to compare the standard Crouzeix-Raviart discretization \eqref{E:CR} and the modified one \eqref{E:CRmod} on sequences of meshes with increasing shape parameter. To this end, we consider the same exact solution as in \S\ref{SS:smooth-sol} and the following sequence $(\CT_n^m)_{n\geq 0}$ of meshes with prescribed anisotropy $m \in \nz$. For any fixed $m$, we divide $\Om$ into $(m2^n) \times 2^n$ rectangles, with edges parallel to the lines $x_2 = 0$ and $x_1 = 0$. The length of the edges is $2^{-n}/m$ and $2^{-n}$, respectively. Then, we obtain $\CT_n^m$ by drawing, for each rectangle, the diagonal parallel to the line $x_2 = m x_1$, see Figure~\ref{F:meshes-ex1-2}. For $m=1$, this is nothing else than the mesh $\CT_n$ considered in the previous experiment. The diameter of all triangles in $\CT_n^m$ is proportional to $2^{-n}$.

Since the maximum angle in all meshes is $\pi/2$, we expect and observe that the velocity $H^1$-error and the pressure $L^2$-error  of the standard discretization (and, consequently, also $e_{\CT_n^m}(\Vu)$ and $e_{\CT_n^m}(p)$) converge to zero with the maximum decay rate $2^{-n} \approx (\#\CT_n^m/m)^{-0.5}$, irrespective of $m$, cf. \cite{AcoDur}. The same result cannot be inferred from the quasi-optimal estimates in Theorems~\ref{T:main} and \ref{T:presserr}, because the shape parameter of $\CT_n^m$ is proportional to $m$. Moreover, it does not appear possible to improve such estimates, because we numerically computed the best constant in \eqref{E:SK-stable} for this type of triangles and observed a linear dependence on $m$. (Recall that this constant enters into the bound of 
$\shps_1$ in Theorem~\ref{T:main}). 

Proceeding as before, we compute the ratios 
\begin{equation*}
\gamma_n^m(\Vu) := \dfrac{\norm[L^2(\Om)]{\nabla \Vu - \nabla_{\CT_n^m} \Vu_{\CT_n^m}}}{e_{\CT^m_n}(\Vu)} 
\qquad \text{and} \qquad
\gamma_n^m(p) := \dfrac{\norm[L^2(\Om)]{p - p_{\CT^m_n}}}{e_{\CT^m_n}(p)} 
\end{equation*}
both for the standard discretization and the modified one. Some values of $\gamma_n^m(\Vu)$ and $\gamma_n^m(p)$ are displayed in Tables~\ref{F:velerr/best} and \ref{F:preerr/best}, respectively, for $m \in \{10, 20, 40\}$. They indicate that, in this specific example, the performance of both discretizations for large $m$ remains close to the best possible and is similar to the one observed in \S\ref{SS:smooth-sol} for $m=1$.  
 
\subsection{Rough pressure}\label{SS:rough-pre}

This experiment aims at illustrating the pressure-robust\-ness of the modified discretization \eqref{E:CRmod}. Hence, we consider a test case with smooth analytical veclocity and rough analytical pressure, namely  
\begin{equation*}\label{E:test3-sol}
\Vu(x_1, x_2) = \Curl\left(\:x_1^2(x_1-1)^2x_2^2(x_2-1)^2\:\right),
\quad
p(x_1, x_2) = \begin{cases}
\frac{\pi}{\pi-1} & \text{if} \;\; x_1 > \pi^{-1} \\
-\pi & \text{if} \;\; x_1 < \pi^{-1}
\end{cases}.
\end{equation*}
We construct the initial triangulation $\widehat{\CT}_0$ by drawing the diagonals of $\Om$, see Figure~\ref{F:meshes-ex3-4}. The intersection of the diagonals is taken as newest vertex for all triangles in $\widehat{\CT}_0$. Each one of the successive meshes $\widehat{\CT}_n$, $n \geq 1$, is obtained from the previous one through two global refinements by newest vertex bisection.

Notice that the load $\Vf = -\Delta \Vu + \nabla p $ has a singular part concentrated on the line $l:= \{ \pi^{-1} \} \times (0,1)$. Thus, $\Vf$ does not belong to $L^2(\Om)^2$. Still, it is possible to extend the standard Crouzeix-Raviart discretization \eqref{E:CR} to this case because, for all $n\geq 0$, each edge of $\widehat{\CT}_n$ intersects $l$ in at most one point.   

Since $\Vu \in H^2(\Om)^2$, Theorem~\ref{T:main} and standard interpolation estimates entail that $e_{\widehat{T}_n}(\Vu)$ and the velocity $H^1$-error of \eqref{E:CRmod} converge to zero with the maximum decay rate $(\#\widehat{\CT}_n)^{-0.5}$. Moreover, our numerical results indicate that the ratio $\gamma_n(\Vu)$, defined as in \eqref{E:error/best} with $\widehat{\CT}_n$ in place of $\CT_n$, is nearly 2. In contrast, the data displayed in Figure~\ref{F:plot-ex3} show that the velocity $H^1$-error of \eqref{E:CR} is impaired by the low regularity of the analytical pressure and converges approximately with rate $(\#\widehat{\CT}_n)^{-0.25}$. The pressure $L^2$-errors of both discretizations are quite close to the best $L^2$-error $e_{\widehat{\CT}_n}(p)$ and converge approximately with decay rate $(\#\widehat{\CT}_n)^{-0.25}$.

\begin{figure}[htp]
	\hfill
	\subfloat{\includegraphics[width=0.4\hsize]{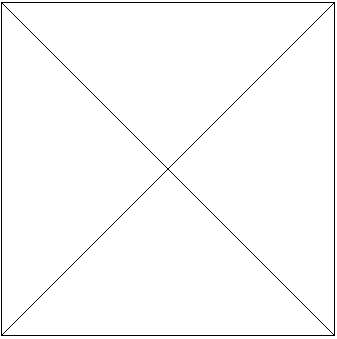}}
	\hfill
	\subfloat{\includegraphics[width=0.4\hsize]{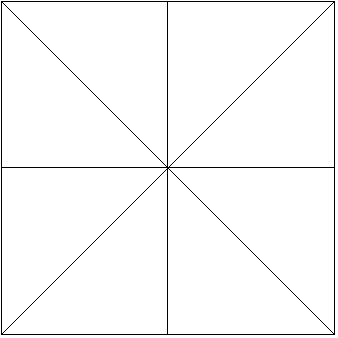}}
	\hfill
	\caption{Initial meshes $\widehat{\CT}_0$ from \S\ref{SS:rough-pre} (left) and $\breve{\CT}_0$ from \S\ref{SS:rough-load} (right).}
	\label{F:meshes-ex3-4}
\end{figure}

\begin{figure}[htp]
	{\includegraphics[width=0.6\hsize]{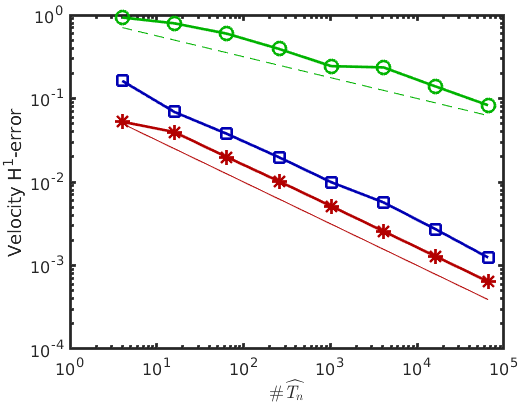}}
	\caption{Convergence histories of the $H^1$-errors of the standard ($\circ$) and the modified ($\square$) discretizations and of the best $H^1$-error to the analytical velocity ($*$) from \S\ref{SS:rough-pre}. Plain and dashed lines indicate decay rates $(\#\widehat{\CT}_n)^{-0.5}$ and $(\#\widehat{\CT}_n)^{-0.25}$, respectively.}
	\label{F:plot-ex3}
\end{figure}

\subsection{Rough load}\label{SS:rough-load}

Similarly as in the previous experiment, we now consider the Stokes equations \eqref{E:Stokes} with a load $\Vf \notin L^2(\Om)^2$. Indeed, we assume that the action of $\Vf$ on $H^1_0(\Om)^2$-functions is given by
\begin{equation*}\label{E:test4-load}
\left\langle \Vf, \Vv \right\rangle 
=
\int_0^1 x_2 \:\Vv(0.5, x_2) \cdot \Vt\: dx_2
\end{equation*}
where $\Vt = (0,1)$. The main difference from \S\ref{SS:rough-pre} is that here $\Vf$ is concentrated on the line $l = \{0.5\} \times (0,1)$ and its density is not constant (nor piecewise constant) in the $1$-dimensional Hausdorff measure.  

Let $\breve{\CT}_n$, $n\geq 0$, be obtained from the mesh $\widehat{\CT}_n$ in \S\ref{SS:rough-pre} through a global refinement by newest vertex bisection, see Figure~\ref{F:meshes-ex3-4}. Since some edges of $\breve{\CT}_n$ are contained in the support $l$ of $\Vf$, the standard Crouzeix-Raviart discretization $\eqref{E:CR}$ cannot be extended to this case. (The same observation applies also to the discretization proposed in \cite{Lin,LinMerNeiNeu18}.) In contrast, the modified discretization of \S\ref{S:mod} is well-defined, because $\Vf \in H^{-1}(\Om)^2$. Thus, we only solve \eqref{E:CRmod}. 

As the analytical solution is not available, we only monitor the difference between the approximations $( \Vu_{\breve{\CT}_{n-1}}, p_{\breve{\CT}_{n-1}})$ and $(\Vu_{\breve{\CT}_n}, p_{\breve{\CT}_n})$, obtained on $\breve{\CT}_{n-1}$ and $\breve{\CT}_n$, respectively. Hence, we compute
\begin{equation*}
\delta_n(\Vu) := \norm[L^2(\Om)]{\nabla_{\breve{\CT}_n} \Vu_{\breve{\CT}_n} - \nabla_{\breve{\CT}_{n-1}} \Vu_{\breve{\CT}_{n-1}} }
\qquad \text{and} \qquad
\delta_n(p) := \norm[L^2(\Om)]{p_{\breve{\CT}_n} - p_{\breve{\CT}_{n-1}}}
\end{equation*}
for $n\geq 1$. We estimate the decay rate of $\delta_n(\cdot)$ in terms of $\#\breve{\CT}_n$ through the so-called experimental order of convergence
\begin{equation*}\label{E:EOC}
\mathrm{EOC}_n(\cdot) := \dfrac{\log\left(\delta_n(\cdot)  \big/ \delta_{n-1}(\cdot)\right)}{\log\left( (\#\breve{\CT}_n) \big/ (\#\breve{\CT}_{n-1})\right)}
\end{equation*} 
for $n \geq 2$. Some values of $\delta_n(\Vu)$ and $\delta_n(p)$, together with the corresponding EOCs, are displayed in Table~\ref{F:rough-load}. 

The load $\Vf$ is in $H^{-1/2-\varepsilon}(\Om)$ for all $\varepsilon > 0$ and this implies that the corresponding analytical solution is in $H^{3/2-\varepsilon}(\Om) \times H^{1/2-\varepsilon}(\Om)$, owing to the shift theorem of \cite{Dau}. Thus, one may expect that $\mathrm{EOC}_n(\Vu)$ and $\mathrm{EOC}_n(p)$ are nearly $0.25$. The values in Table~\ref{F:rough-load} actually indicate a higher decay rate, both of $\delta_n(\Vu)$ and $\delta_n(p)$. A possible explanation is that, focusing for instance on the best $H^1$-error to the analytical velocity $\Vu$, one has, for all $n \geq 0$,
\begin{equation*}
e_{\breve{\CT}_n}(\Vu)^2 =
\inf_{\Vv_{\breve{\CT}_n} \in CR(\breve{\CT}_n)} 
\norm[L^2(\Om_{-})]{\nabla \Vu - \nabla_{\breve{\CT}_n} \Vv_{\breve{\CT}_n}}^2
+ \inf_{\Vv_{\breve{\CT}_n} \in CR(\breve{\CT}_n)} 
\norm[L^2(\Om_{+})]{\nabla \Vu - \nabla_{\breve{\CT}_n} \Vv_{\breve{\CT}_n}}^2
\end{equation*} 
where $\Om_{-} = (0, 0.5)\times (0,1)$ and $\Om_{+} = (0.5, 1) \times(0,1)$, cf. \cite[eq. (1.2)]{VeeZan17b}. Therefore, according to Theorem~\ref{T:main}, the modified discretization \eqref{E:CRmod} is potentially able to exploit additional regularity of $\Vu_{|\Om_{-}}$ and $\Vu_{|\Om_{+}}$ beyond the one of $\Vu$. 

\begin{table}
	\begin{tabular}{r|r|cc|cc}
		n & $\#\breve{\CT}_n$ &
		$\delta_n(\Vu)$  & $\mathrm{EOC}_n(\Vu)$ &
		$\delta_n(p)$  & $\mathrm{EOC}_n(p)$
		\\[1ex]
		\hline
		&&&
		\\[-1.5ex]
		3 &  512 &
		6.092e-02 &  & 
		4.339e-02 & 
		\\
		4 & 2048 &
		3.673e-02 & \raisebox{1.5ex}[0pt]{0.37} & 
		2.571e-02 & \raisebox{1.5ex}[0pt]{0.38}
		\\
		5 & 8192 &
		2.135e-02 & \raisebox{1.5ex}[0pt]{0.39} & 
		1.455e-02 & \raisebox{1.5ex}[0pt]{0.41}
		\\
		6 & 32768 &
		1.206e-02 & \raisebox{1.5ex}[0pt]{0.41} & 
		8.021e-03 & \raisebox{1.5ex}[0pt]{0.43}
		\\
		7 & 131072 &
		6.670e-03 & \raisebox{1.5ex}[0pt]{0.43} & 
		4.349e-03 & \raisebox{1.5ex}[0pt]{0.44}
	\end{tabular}
	\vspace{1ex}
	\caption{Values $\delta_n(\Vu)$ and $\delta_n(p)$ from \S\ref{SS:rough-load} and corresponding EOCs.}
	\label{F:rough-load}
\end{table}
%
%
\subsection*{Acknowledgements}
We thank Andreas Veeser for many inspiring discussions and the unknown reviewers for their comments which led to substantial improvements of the present article.
%
%

\begin{thebibliography}{10}

\bibitem{AcoDur}
\sc G.~Acosta and R.~G. Dur\'an, {\em The maximum angle condition for mixed and nonconforming elements: application to the {S}tokes equations}, SIAM J.
Numer. Anal., 37 (1999), pp.~18--36.

\bibitem{Bog}
\sc M.~E. Bogovski\u\i, {\em Solution of the first boundary value problem for an equation of continuity of an incompressible medium}, Dokl. Akad. Nauk SSSR, 248 (1979), pp.~1037--1040.

\bibitem{Bren14}
\sc S.~C. Brenner, {\em Forty years of the {C}rouzeix-{R}aviart element}, Numer. Methods Partial Differential Equations 31 (2015), no.~2, 367--396.

\bibitem{BreSco}
S.~C. Brenner and L.~R. Scott, {\em The {M}athematical {T}heory of {F}inite {E}lement {M}ethods}, vol.~15 of Texts in Applied Mathematics, Springer, New York, third~ed., 2008.

\bibitem{BreFor}
F.~Brezzi and M.~Fortin, \emph{Mixed and {H}ybrid {F}inite {E}lement
  {M}ethods}, Springer {S}eries in {C}omputational {M}athematics, vol.~15,
  Springer, Berlin, 1991.

\bibitem{CarGalNat}
C.~Carstensen, D.~Gallistl, and N.~Nataraj, {\em Comparison results of nonstandard {$P_2$} finite element methods for the biharmonic problem}, ESAIM
Math. Model. Numer. Anal., 49 (2015), pp.~977--990.

\bibitem{CarSch}
C.~Carstensen and M.~Schedensack, {\em Medius analysis and comparison results for first-order finite element methods in linear elasticity}, IMA J. Numer. Anal., 35 (2015), pp.~1591--1621.

\bibitem{CroRav}
M.~Crouzeix and P.-A. Raviart, {\em Conforming and nonconforming finite
	element methods for solving the stationary {S}tokes equations. {I}}, Rev.
Fran\c{c}aise Automat. Informat. Recherche Op\'erationnelle S\'er. Rouge, 7
(1973), pp.~33--75.

\bibitem{Dau}
{\sc M.~Dauge}, {\em Stationary {S}tokes and {N}avier--{S}tokes systems on two- or three-dimensional domains with corners. {P}art {I}: {L}inearized equations}, SIAM J. Math. Anal., 20 (1989), pp.~74--97.

\bibitem{GuzNei14}
J.~Guzm\'an and M.~Neilan, {\em Conforming and divergence-free {S}tokes
	elements on general triangular meshes}, Math. Comp., 83 (2014), pp.~15--36.

\bibitem{GuzNei17}
\bysame, \emph{Inf-sup stable finite elements on barycentric refinements producing divergence--free approximations in arbitrary dimensions}, SIAM J. Numer. Anal., 56 (2018), pp.~2826--2844.

\bibitem{Hec}
F.~Hecht, {\em Construction d'une base de fonctions {$P_{1}$} non
	conforme \`a divergence nulle dans {${\bf R}^{3}$}}, RAIRO Anal. Num\'er., 15
(1981), pp.~119--150.

\bibitem{HeiKoeKriSchSie}
{\sc C.-J. Heine, D.~K\"oster, O.~Kriessl, A.~Schmidt, and K.~Siebert}, {\em
	{ALBERTA}: an adaptive hierarchical finite element toolbox}.
\newblock accessed \today, http://www.alberta-fem.de.

\bibitem{Lin}
A.~Linke, \emph{On the role of the {H}elmholtz decomposition in mixed methods for incompressible flows and a new variational crime}, Comput. Methods Appl. Mech. Engrg., 268 (2014), pp.~782--800.

\bibitem{LinMerNeiNeu18}
A.~Linke, C.~Merdon, M.~Neilan, and F.~Neumann, {\em Quasi-optimality of
	a pressure-robust nonconforming finite element method for the
	{S}tokes-problem}, Math. Comp., 87 (2018), pp.~1543--1566.


\bibitem{Qin}
J.~Qin, {\em On the convergence of some simple finite elements for incompressible flows}, PhD thesis, Penn State University, 1994.

\bibitem{SchSie}
{\sc A.~Schmidt and K.~G. Siebert}, {\em Design of adaptive finite element software}, vol.~42 of Lecture Notes in Computational Science and Engineering, Springer-Verlag, Berlin, 2005.

\bibitem{ScoVoga}
L.~R. Scott and M.~Vogelius, {\em Conforming finite element methods for incompressible and nearly incompressible continua}, in Large-scale computations in fluid mechanics, {P}art 2 ({L}a {J}olla, {C}alif., 1983), vol.~22 of Lectures in Appl. Math., Amer. Math. Soc., Providence, RI, 1985, pp.~221--244.

\bibitem{ScoVogb}
\bysame, {\em Norm estimates for a maximal right inverse of the divergence operator in spaces of piecewise polynomials}, RAIRO Mod\'el. Math. Anal. Num\'er., 19 (1985), pp.~111--143.

\bibitem{VeeZan17a}
A.~Veeser and P.~Zanotti, {\em Quasi-optimal nonconforming methods for
	symmetric elliptic problems. {I}---{A}bstract theory}, SIAM J. Numer. Anal.,
56 (2018), pp.~1621--1642.

\bibitem{VeeZan17b}
\bysame, \emph{Quasi-optimal nonconforming methods for symmetric elliptic
  problems. {II} -- {O}verconsistency and classical nonconforming elements},
  SIAM J. Numer. Anal., to appear.

\bibitem{VeeZan17c}
\bysame, \emph{Quasi-optimal nonconforming methods for symmetric elliptic problems. {III}---{D}iscontinuous {G}alerkin and other interior penalty methods}, SIAM J. Numer. Anal., 56 (2018), pp.~2871--2894.

\bibitem{XuZha}
X.~Xu and S.~Zhang, {\em A new divergence-free interpolation operator
	with applications to the {D}arcy-{S}tokes-{B}rinkman equations}, SIAM J. Sci.
Comput., 32 (2010).

\bibitem{Zha05}
S.~Zhang, {\em A new family of stable mixed finite elements for the 3{D} {S}tokes equations}, Math. Comp., 74 (2005), pp.~543--554.

\bibitem{Zha08}
\bysame, {\em On the {P}1 {P}owell-{S}abin divergence-free finite element for the {S}tokes equations}, J. Comput. Math., 26 (2008), pp.~456--470.

\end{thebibliography}
%
\providecommand{\bysame}{\leavevmode\hbox to3em{\hrulefill}\thinspace}
\providecommand{\MR}{\relax\ifhmode\unskip\space\fi MR }
\providecommand{\MRhref}[2]{%
  \href{http://www.ams.org/mathscinet-getitem?mr=#1}{#2}
}
\providecommand{\href}[2]{#2}

\end{document}